\documentclass[a4paper,12pt]{amsart}

 \usepackage{amssymb,amsthm}
 \usepackage{enumerate} 
\usepackage{diagbox}

\usepackage{amscd} 
\usepackage{graphicx,color} 
\numberwithin{equation}{section}

\usepackage{url} 

\newtheorem{theorem}{Theorem}[section]

\newtheorem{proposition}[theorem]{Proposition}

\theoremstyle{definition}

\newtheorem{example}[theorem]{Example}
\newtheorem{remark}[theorem]{Remark}

\usepackage{relsize}

\def\r{\mathbb R}

\def\r{\mathbb R}
\def\rr{\mathbb L}
 
\def\t{\mathbf t}
\def\n{\mathbf n}

\def\b{\mathbf b}

\begin{document}

\title{Maximal translation surfaces in Lorentz-Minkowski space }
\author{Rafael L\'opez}
\address{ Departament of Geometry and Topology\\ University of Granada. 18071 Granada, Spain}
\email{rcamino@ugr.es}

\begin{abstract} A translation surface in Lorentz-Minkowski space $\rr^3$ is a surface defined as the sum of two spatial curves. In this paper we present a classification of maximal surfaces of translation type.  We prove that if a generating curve is planar, then    the other generating curve is also planar.  We give a full description of these surfaces. In case that both curves are of Frenet type, we generalize the Scherk surfaces. In case that a curve is a pseudo-null curve, we obtain new examples of maximal surfaces which have not counterparts in Euclidean space.
\end{abstract}

\maketitle

\section{Introduction  and statement of the result}
 
  Let $z\colon\Omega\to\r$ be a smooth function, $z=z(x,y)$, defined in a domain $\Omega$ of $\r^2$. The graph $z=z(x,y)$     is a minimal surface in $\r^3$ if  $z$ satisfies 
 $$(1+z_y^2)z_{xx}-2z_xz_yz_{xy}+(1+z_x^2)z_{yy}=0.$$
If we study the solutions of this equation by separation of variables, $z(x,y)=f(x)+g(y)$, for smooth functions $f\colon I\subset\r\to\r$ and $g\colon J\subset\r\to\r$, then the minimal surface equation reduces into
$$(1+g'^2)f''=(1+f'^2)g''.$$
This equation has only two types of solutions. First, both functions $f$ and $g$ are linear functions in its variable and  the surface is a plane. The second solution is 
\begin{equation}\label{sc}
z(x,y)=\frac{1}{c}\left(\log \cos(cx)-\log \cos(cy)\right),\quad c>0.
\end{equation}
This surface is called the Scherk surface \cite{sc}.  

A surface $z=f(x)+g(y)$ can be also viewed as the sum of two curves of $\r^3$, namely, $x\mapsto (x,0,f(x))$ and $y\mapsto (0,y,g(y))$. More generally, a surface $M$ in $\r^3$ is said to be a {\it translation surface} if $M$ is the sum of two  curves of $\r^3$.  The notion of translation surface is not metric but affine. Notice that the   two curves that generate the Scherk surface are contained in orthogonal planes. In \cite{di},  the authors proved that if a generating curve of a translation minimal surface is contained in a plane then the other is also planar. However, both curves are contained in planes which are not necessarily orthogonal. These surfaces belong to a large family of surfaces discovered by Scherk in XIXth century, where \cite{sc} is a particular example: see also \cite{ni}. Since the notion of translation surface is not metric but it only requires of an operation of group in the space, the study of translation surfaces with zero mean curvature, and also with constant mean curvature, has been extended to other ambient spaces. Without to give a full list of references, we refer to the reader to \cite{ay,hl2,Lopez,lomu,mi}.

Until very recently, the Scherk surfaces were the only known minimal translation surfaces in Euclidean space. 
However, the author in collaboration with O. Perdomo (CCSU) and T. Hasanis (Ioannina), found a plethora of new examples of minimal translation surfaces  \cite{hl,lp}. In fact,   a method was given to obtain all minimal translation surfaces.  An interesting example is the helicoid which it is obtained  by sum of a circular helix by itself. For this surface, the generating curves are not planar.

In this paper, we consider the analogue problem in Lorentz-Minkowski space $\rr^3$. The space $\rr^3$ is  the vector space  $\r^3$  endowed with the metric $\langle,\rangle =dx^2+dy^2-dz^2$, where $(x,y,z)$ denote the standard coordinates of $\r^3$. As in Euclidean space, we can consider  translation surfaces in $\rr^3$, that is, surfaces that are the sum of two spatial curves  $\alpha:I\subset\r \rightarrow\rr^3$, $\beta:J\subset\r\rightarrow\rr^3$ and whose parametrization is
\begin{equation}\label{tt}
X(s,t)=\alpha(s)+\beta(t).
\end{equation}
Regularity of the surface is equivalent to $\alpha^\prime(s)\times\beta'(t)\not=0$ for all $(s,t)\in I\times J$. In $\rr^3$ we consider spacelike surfaces, that is, surfaces whose   induced metric is Riemannian. Spacelike surfaces with zero mean curvature $H=0$ at every point are called maximal surfaces because they maximize locally the surface area functional. 

First  examples of maximal translation surfaces are (spacelike) planes. Similarly as in $\r^3$, it is posible to find maximal surfaces of $\rr^3$ of type $z=f(x)+g(y)$, although it is also necessary to consider surfaces of type  $y=f(x)+g(y)$ (or $x=f(y)+g(z)$). This is   because the $xy$-plane  is a spacelike plane,  but the $xz$-plane  and the $yz$-plane are timelike. These surfaces were obtained in \cite{wo} and they are the analogs of the Scherk surface \eqref{sc}, namely, 
\begin{equation}\label{sc2}
\begin{split}z(x,y)&=\frac{1}{c}\left(\log \cosh{cx}-\log \cosh{cy}\right)=\frac{1}{c} \log \frac{\cosh{cx}}{\log \cosh{cy}},\quad c>0,\\
y(x,z)&=\frac{1}{c}\left(\log \sinh{cz}-\log\cos{cx}\right)=\frac{1}{c} \log \frac{\sinh{cz}}{\log \cos{cz}},\quad c>0.
\end{split}
\end{equation}
 However, the question of classification of all maximal translation  surfaces in $\rr^3$ is not yet answered. For instance, until now no  examples are known when the generating curves are not planar. Even in case that the generating curves were planar, no examples are known if the planes are not orthogonal, neither examples when one (or two) generating curves are contained in lightlike planes. Notice that a spacelike planar curve can be contained in a timelike or in a lightlike plane.  In this paper we give new progress in this problem, obtaining new examples of maximal translation  surfaces and thus, we enrich with  new examples   the class of maximal surfaces of $\rr^3$. 

An interesting observation is that, contrary to what one might think, no definitive results are obtained by simply repeating the same arguments as in Euclidean space. In order to employ the same techniques than in \cite{hl,lp}, Frenet equations of curves have to be used. Here it appears a  difference with the Euclidean space. For spacelike curves of $\rr^3$ where the normal vector is spacelike or timelike, the Frenet equations are similar to that of $\r^3$, after the corresponding changes of signs and similar notions of curvature $\kappa$ and torsion $\tau$. These curves are called {\it curves of Frenet type}. However, in $\rr^3$ the normal vector of a spacelike curve can be lightlike and the Frenet equations are completely different.  These curves are called {\it pseudo-null curves}.  Even under the situation that both curves are of Frenet type, it is easy to realize that  it is not possible to do, step-by-step, the same arguments than in Euclidean space.  The reason is that in \cite{hl,lp} it is used strongly that self-adjoint maps in $\r^3$ are diagonalizable.  In contrast, self-adjoint maps in $\rr^3$ may be not diagonalizable.  This fact shows that the problem of classification of maximal translation  surfaces in $\rr^3$ is not a simple `change of sign', as sometimes happens when a problem of differential geometry is moved from the Euclidean to the Lorentzian ambient space. In consequence, we have the following
\begin{quote}{\sc Open problem.} Give a classification/description of all maximal translation surfaces in Lorentz-Minkowski space.
\end{quote}

We now present the results of this paper in the context of the above classification problem.  
The main result is the extension of  \cite{di}. 

\begin{theorem} \label{t1}
Let $M$ be a maximal translation  surface in $\rr^3$. If a generating curve is planar, then the other generating curve is also planar.  
\end{theorem}

This result holds for all types of causal character of the  normal vectors of the curves. Recall that a pseudo-null curve is contained in a plane. Therefore, Theorem \ref{t1} solves and completes the classification problem proposed previously when one generating curve is a   pseudo-null curve. 

After the proof of Thm. \ref{t1}, the second goal of this paper is the description of  all these surfaces. When both generating curves are of Frenet type, we obtain   the analogues of the Scherk minimal surfaces. 

\begin{theorem}[Scherk type surfaces]\label{t2}  Suppose that  $M$ is a maximal translation surface generated by two spacelike curves of Frenet type. Then the surface is a plane or, after a rigid motion of $\rr^3$, the surface is one of the following types:
\begin{align}
\Psi(s,t)&=\left(s+t\cosh\theta, \frac{1}{c}\log\frac{\cos(cs)}{\cos(ct)},t\sinh\theta\right),\label{para1}\\
  \Psi(s,t)&=\left(s+t\cos\theta,-t\sin\theta,\frac{1}{c} \log\frac{\cosh cs}{\cosh ct}\right),\label{para2}\\
  \Psi(s,t)&=\left(s+t \sinh (\theta ),-\frac{1}{c}\frac{\log (\sinh (c t))}{\log (\cos (c s))},t \cosh (\theta )\right),\label{para3}
  \end{align}
 where $c>0$ and $\theta\in\r$. These surfaces are called of  Scherk type.  
  \end{theorem}
Notice that if $\theta=0$ in \eqref{para2} and \eqref{para3}, we obtain the surfaces \eqref{sc2}, while if $\theta=0$ in \eqref{para1} we obtain  the plane $z=0$.

When a generating curve is pseudo-null curve, we also obtain the parametrizations of the surfaces: see the explicit parametrizations in Sect. \ref{s3}.

By the variety of cases according to the causal character of the normal vector of the generating curves,  the organization of the paper is the following. In Sect. \ref{s3} we study the case that a generating curve is a pseudo-null curve. The case that  a generating curve is of Frenet type is considered in Subsect. \ref{s31}, otherwise    in Subsect. \ref{s32}. The parametrizations of the surfaces are obtained in Thms. \ref{t31}, \ref{t32} and \ref{t34} .

In Sects. \ref{s4} and \ref{s5}, we study the case that both generating curves are of Frenet type.  In Thm. \ref{t41} we prove that both  generating curves satisfy that   $\kappa^2\tau$ is a constant function. An interesting case is when a generating curve is a circular helix, proving that  the other curve is a rigid motion of the first one (Thm. \ref{t42}).  This gives examples of translation maximal surfaces generating by non-planar curves. Section \ref{s5} is devoted to the case that one generating curve is planar. In Thm. \ref{t43}, it is proved that if a generating curve is planar, then the other is also planar.  This result together that of Sect. \ref{s3} proves Thm. \ref{t1}. In the rest of Sect. \ref{s5} we obtain parametrizations of all these  surfaces, which appear in Thm. \ref{t2}. These surfaces are generated by the curves that appears in \eqref{sc2} but where to one of them is applied  a rigid motion of $\rr^3$.

Part of the results can ben straightforward extended to timelike surfaces of $\rr^3$ with zero mean curvature.  This is due by two reasons. First, a translation surface of $\rr^3$ given by \eqref{tt} has, in general, regions where the surface is Riemannian and other ones where the surface is timelike. This occurs regardless the causal character of the generating curves. In consequence, in all parametrizations of translation maximal surfaces (as for example in Thm. \ref{t2}), the surfaces may have open regions where the surface is timelike, but the mean curvature follows being zero.

  A second reason is that for timelike surfaces the notion of mean curvature as the trace of the Weingarten map coincides for spacelike   surfaces. The computation of the  equation $H=0$ by using a given parametrization of the surface is the same regardless if the surface is spacelike or timelike, even knowing that the Weingarten map of a timelike surface can be not diagonalizable.

\section{Preliminaries}\label{s2}

In this section we recall some basics of the differential geometry of curves and surfaces in $\rr^3$ \cite{lo2} and we obtain the  Scherk surfaces \eqref{sc2}. A vector $v\in\rr^3$  is said to be spacelike (resp. timelike or lightlike) if $\langle v,v\rangle>0$ or $v=0$ (resp. $\langle v,v\rangle<0$, $\langle v,v\rangle=0$). The vector product of two vectors $u,v$ is defined as the unique vector $u\times v$ such that $\langle u\times v,w\rangle=(u,v,w)$,  for all $w\in\rr^3$, where $(u,v,w)$ denotes the determinant   $\mbox{det}(u,v,w)$. 
   A curve $\alpha\colon I\subset\r\to\rr^3$  is said to be spacelike (resp. timelike, lightlike) if $\alpha'(s)$ is spacelike (resp. timelike, lightlike) for all $s\in I$. 
 
 \begin{remark} All our results of classification are local and thus, we will assume that  the causality of $\alpha$, as well as its derivatives, is the same for all $s\in I$. Notice that the spacelike and the timelike conditions are open in the sense that if $\alpha$ is spacelike (or timelike) at $s=s_0$, then in an interval around $s_0$ the curve $\alpha$ is also spacelike (or timelike). In case that $\alpha''$ is lightlike, we will assume that this occurs in the domain $I$ of $\alpha$. Unless otherwise specified, all spacelike curves will be parametrized by arc length. 
\end{remark}
  
We recall the Frenet frame and the Frenet equations of a spacelike curve. Suppose that $\alpha\colon I\subset \r\to\rr^3$, $\alpha=\alpha(s)$, is a spacelike curve of $\rr^3$ parametrized by arc-length. The vector $\t(s)=\alpha'(s)$ is called the tangent vector to $\alpha$.  Suppose that $\t'(s) \not=0$  all $s\in I$. We distinguish two cases according the causality of $\alpha''$.

\begin{enumerate}
\item Case   $\alpha''$ is spacelike or timelike. We say that $\alpha$ is a {\it curve of Frenet type}. The curvature is defined by $\kappa=|\t'|$ and the unit normal is $\n(s)=\t'(s)/\kappa(s)$. The binormal is defined by  $\b=\t\times\n$. Let $\langle\n,\n\rangle=\epsilon\in\{-1,1\}$ and thus $\langle \b,\b\rangle=-\epsilon$.  The Frenet equations are
\begin{equation}\label{f1}
\left\{\begin{array}{ccccc}
\t'&=& &\kappa \n & \\
\n'&=&-\epsilon\kappa \t& +& \tau \b\\
\b'&=& & \tau \n. & \\
\end{array}\right.
\end{equation}
 
\item Case  $\alpha''$ lightlike. These curves are called {\it pseudo-null curves}. The normal vector is defined by  $\n=\t'$. For each $s\in I$, define $\b(s)$ are the unique  lightlike vector such that $\b(s)$ is orthogonal to $\t(s)$, $\langle \n(s),\b(s)\rangle=1$ and $(\t(s),\n(s),\b(s))=1$ for all $s\in I$.   In particular, notice that  $\t\times\n=-\n$. The Frenet equations are
\begin{equation}\label{f3}
\left\{\begin{array}{ccccc}
\t'&=& & \n & \\
\n'&=&& \kappa \n&  \\
\b'&=&-\t &  &-\kappa \b. \\
\end{array}\right.
\end{equation}

\end{enumerate}

Pseudo-null curves are curves contained in lightlike planes \cite{da}. If  $\kappa$ is constant, we have explicit parametrizations  of these curves.  

\begin{proposition}\label{pr1} Let $\alpha\colon I\to\rr^3$ be a spacelike curve parametrized by arc-length such that $\alpha$ is a pseudo-null curve. If $\kappa$ is constant, then $\alpha$ can be parametrized by
\begin{equation}\label{pr11}
\begin{split}
\alpha(s)&=\frac{s^2}{2}\vec{v}+s\vec{b},\quad (\kappa=0),\\
\alpha(s)&=e^{ks}\vec{v}-\frac{s}{k}\vec{b},\quad (\kappa\not=0),
\end{split}
\end{equation}
where $\vec{v}$ is a lightlike vector, $\vec{b}$ is a unit spacelike vector and orthogonal to $\vec{v}$.
\end{proposition}

\begin{proof}
If $\kappa=0$ identically, then $\n$ is constant. The derivative of  the function $s\mapsto \langle\alpha(s),\n\rangle$ is $0$, proving that $\alpha$ is contained in a plane. We can parametrize $\alpha$ by integrating $\t''(s)=0$, obtaining $\alpha(s)=\frac{s^2}{2}\vec{v}+s\vec{b}$, where $\vec{v}$ is a lightlike vector, $\vec{b}$ is a unit spacelike vector and orthogonal to $\vec{v}$.

Suppose $\kappa\not=0$, $\kappa=k$.   Then $\t''=k\t'$. This gives $\t'=k\t+\vec{b}$ for some   vector $\vec{b}\in\rr^3$.  Integrating again, and after a translation, we  get 
$$\alpha(s)=e^{ks}\vec{v}-\frac{s}{k}\vec{b},$$
where $\vec{v}\in\rr^3$. 
Identity $|\alpha'(s)|^2=1$ becomes
$$k^2e^{2ks}|\vec{v}|^2-2e^{ks}\langle\vec{v},\vec{b}\rangle+\frac{1}{k^2}|\vec{b}|^2=1$$
for all $s\in I$. Since the functions $\{1,e^{ks}, e^{2ks}\}$ are linearly independent, then $\vec{v}$ is a lightlike vector orthogonal and $\vec{b}$ is spacelike with $|\vec{b}|^2=k^2$ and orthogonal to $\vec{v}$.
\end{proof}

A surface $M\subset\rr^3$ is called spacelike (resp. timelike) if the induced metric on $M$ is Riemannian (resp. Lorentzian) at every tangent plane. As a consequence, all curves contained in a spacelike surface are spacelike. Given a spacelike surface, it is  possible to choose a unit normal vector field $N$  globally defined on  $M$, being $N$ timelike at every point.  The  shape operator $S=-(dN)$  is diagonalizable and its eigenvalues are the principal curvatures $\kappa_1$, $\kappa_2$. The mean curvature is defined by $H=-\frac12(\kappa_1+\kappa_2)$ and the Gauss curvature is $K=-\kappa_1\kappa_2$. A {\it maximal surface} is a spacelike surface with $H=0$ at every point.

If $X=X(u,v)$ is a parametrization of a spacelike surface, the coefficients of the first fundamental form are $E=\langle X_u,X_u\rangle$, $F=\langle X_u,X_v\rangle$ and $G=\langle X_v,X_v\rangle$. If we take the unit normal vector $N=\frac{X_u\times X_v}{|X_u\times X_v|}$, the condition $H=0$ writes as 
\begin{equation}\label{h}
G(X_u,X_v,X_{uu})-2F(X_u,X_v,X_{uv})+E(X_u,X_v,X_{vv})=0.
\end{equation}

\begin{remark}  For timelike surfaces, it is also possible to define the mean curvature $H$ begin now $2H$ the trace of $S$. In this case, the shape operator $S$ may be not diagonalizable but formula \eqref{h} is still valid for timelike surfaces with zero mean curvature at every point.    
\end{remark}

Let $M$ be   a  translation surface parametrized by \eqref{tt}. Since $X_{st}=0$, the computation of  $H=0$ in \eqref{h} is  
\begin{equation}\label{eq1}
 (\t_\alpha,\t_\beta,\t_\alpha')+(\t_\alpha,\t_\beta,\t_\beta')=0.
\end{equation}
Since the generating curves  are parametrized by arc-length, the principal curvatures of $M$ are $\kappa_1=(\t_\alpha,\t_\beta,\t_\alpha')$ and $\kappa_2=(\t_\alpha,\t_\beta,\t_\beta')$.

 In the results of this paper we use that spacelike surfaces of $\rr^3$ where  both $H$ and $K$ are constant (isoparametric surfaces) are  planes, hyperbolic planes and circular cylinders (see a proof in \cite{lo22}). In consequence, a maximal surface where a principal curvature is identically $0$,  is contained  in a plane. 
 
 Let $M$ be a maximal translation surface where a generating curve is straight-line. If, say, $\alpha$ is a straight-line, then $\t_\alpha'=0$, which implies that $\kappa_1=0$ on the surface, hence the surface is included in a plane.   From now, we will discard this situation and all generating curves of maximal translation surfaces will not be straight-lines. In particular, the Frenet equations \eqref{f1} and \eqref{f3} hold for both curves.
 
 To finish this section, we go back to the Scherk surfaces given in \eqref{sc2}. We analyze what type of curves are their generating curves which will be  useful in the next arguments. For this we need to know the parametrization by arc-length of these curves. 
\begin{enumerate}
\item Let $\alpha(x)=(x,\frac{1}{c}\log(\cos (c x)),0)$. Its parametrization by arc-length is 
\begin{equation}\label{cu1}
\alpha(s)=\frac{1}{c}\left(\tan^{-1}\sinh(cs),\log\cosh(cs),0\right).
\end{equation}
The unit normal is spacelike and the curvature is  $\kappa(s)=\frac{c}{\cosh(cs)}$. 

\item Let $\alpha(x)=(x,0,\frac{1}{c}\log(\cosh (c x)))$. Its parametrization by arc-length is
\begin{equation}\label{cu2}
\alpha(s)=\frac{1}{c}\left(\tanh^{-1}\sinh(c s),0,\log\cos(cs)\right).
\end{equation}
The unit normal is timelike and the curvature is   $\kappa(s)=\frac{c}{\cos(cs)}$. 

\item Let $  \alpha(z)=(0,\frac{1}{c}\log(\sinh (cz),z)$. Its parametrization by arc-length is 
\begin{equation}\label{cu3}
\alpha(s)=\frac{1}{c}\left(\log\sinh(c s),0,\log\coth\frac{cs}{2}\right).
\end{equation}
The unit normal is timelike and the curvature is   $\kappa(s)=\frac{c}{\sinh(cs)}$. 
\end{enumerate}
  
\section{Case where a generating curve is a pseudo-null curve}\label{s3}

Consider a maximal translation surface $M$ given by the parametrization \eqref{tt}. In this section we will assume that a generating curve, for example $\beta$ is a pseudo-null curve which.  We separate the cases that $\alpha$ is or is not a Frenet curve.

\subsection{Case $\alpha$ is a Frenet curve}\label{s31} 

Suppose that $\alpha$ is a Frenet curve. We distinguish the cases that $\alpha''$ is spacelike and that $\alpha''$ is timelike.

\begin{theorem} \label{t31}
 Let $M$ be a spacelike  translation surface \eqref{tt} such that  $\alpha''$  is spacelike and  $\beta$ is a pseudo-null curve. If $M$ is a maximal surface, then $M$ is a plane or both curves are planar curves. In the latter case,   after a rigid motion, the surface can be parametrized by 
\begin{equation}\label{e11}
X(x,t)=(x,f(x),0)+\beta(t),
\end{equation}
where  
\begin{equation}\label{e12}
 \begin{split}
 f'(x)&=-\tan\left(k(\cos\theta f(x)-x\sin\theta)+a\right), \\ 
 \beta(t)&=me^{kt}(-\sin\theta,\cos\theta,1)+t(\cos\theta,\sin\theta,0),
 \end{split}
 \end{equation}
 and $a,k,m,\theta\in\r$, $k,m\not=0$.
   \end{theorem}

\begin{proof}
Recall that $\beta$ is included in a plane because $\beta$ is a pseudo-null curve. Frenet equations \eqref{f1} and the identity  $\t_\beta\times\n_\beta=-\n_\beta$ imply that Eq. \eqref{eq1} becomes
\begin{equation}\label{c1}\kappa_\alpha\langle \b_\alpha,\t_\beta\rangle-\langle\t_\alpha,\n_\beta\rangle=0.
\end{equation}
The principal curvatures of $M$ are $\kappa_1=\kappa_\alpha\langle \b_\alpha,\t_\beta\rangle$ and $\kappa_2=-\langle\t_\alpha,\n_\beta\rangle$. Differentiating with respect to  $t$, we have
$$\kappa_\alpha\langle\b_\alpha,\n_\beta\rangle-\kappa_\beta\langle\t_\alpha,\n_\beta\rangle=0.$$
From   \eqref{c1}, we have 
\begin{equation}\label{c11}
\kappa_\beta\langle\b_\alpha,\t_\beta\rangle-\langle\b_\alpha,\n_\beta\rangle=0.
\end{equation}
Differentiating with respect to $t$ again, the equations of Frenet \eqref{f3} give
\begin{equation}\label{k21}\kappa_\beta'\langle\b_\alpha,\t_\beta\rangle=0.
\end{equation}
If $\langle\b_\alpha,\t_\beta\rangle=0$ identically, then the principal curvature $\kappa_1$ of $M$ is identically $0$ and   thus $M$ is a plane. 

Suppose that  $\langle\b_\alpha,\t_\beta\rangle\not=0$ at some point $(s_0,t_0)\in I\times J$. Then \eqref{k21} implies that  in an interval around $s_0$, the curvature $\kappa_\beta$ is constant.   We distinguish if $k$ is $0$ or not.
 
\begin{enumerate}
\item Case $k=0$. By  \eqref{pr11} we know that 
$$\beta(t)=\frac{t^2}{2}\vec{v}+\vec{b}t,$$
where $\vec{v},\vec{b}\in\rr^3$,  $\vec{v}$ is a spacelike vector and $\vec{b}$ is a unit spacelike vector orthogonal to $\vec{v}$.   Putting in \eqref{c1}, we have 
\begin{equation}\label{c8}
\kappa_\alpha\langle\b_\alpha,\vec{v}\rangle t+\kappa_\alpha\langle\b_\alpha,\vec{b}\rangle-\langle\t_\alpha,\vec{v}\rangle=0.
\end{equation}
This is a polynomial on $t$, hence $\langle\b_\alpha,\vec{v}\rangle=0$. Thus $\vec{v}$ is a linear combination of $\t_\alpha$ and $\n_\alpha$. Since $\vec{v}$ is lightlike and $\{\t_\alpha,\n_\alpha\}$ are orthonormal spacelike vectors, this combination is trivial, so $\vec{v}=0$. This gives a contradiction.
\item Case $k\not=0$. From the proof of   \eqref{pr11}, we know that 
\begin{equation}\label{c12}
\beta(t)=e^{kt}\vec{v}-\frac{t}{k}\vec{b},
\end{equation}
with $\vec{v}$  is lightlike orthogonal to $\vec{b}$ and $\vec{b}$ is a spacelike vector with $|\vec{b}|=k$. 
Substituting in \eqref{c1},  we have 
$$k e^{kt}\left(\kappa_\alpha \b_\alpha -k \t_\alpha,\vec{v}\rangle\right)-\frac{\kappa_\alpha}{k}\langle\b_\alpha,\vec{b}\rangle=0.$$
Therefore
\begin{equation}\label{c9}
\begin{split}
0&= \langle \kappa_\alpha \b_\alpha-k\t_\alpha,\vec{v}\rangle,\\
0&=\frac{\kappa_\alpha}{k}\langle\b_\alpha,\vec{b}\rangle.
\end{split}
\end{equation}
From the second equation, we have $\langle\b_\alpha,\vec{b}\rangle=0$. Differentiating, we get  $\tau_\alpha\langle\n_\alpha,\vec{b}\rangle=0$. We distinguish two cases.
\begin{enumerate}
\item Case $\langle\n_\alpha,\vec{b}\rangle=0$ identically. Differentiating again, we have $\kappa_\alpha\langle\t_\alpha,\vec{b}\rangle=0$. Since $\kappa_\alpha\not=0$, then $\langle \t_\alpha,\vec{b}\rangle=0$. This gives $\vec{b}=0$, a contradiction. 
\item Case $\tau_\alpha=0$.   In particular, $\alpha$ is a planar curve. This proves Thm. \ref{t1} in this situation. After a rigid motion, we can assume that $\alpha$ is contained in the $xy$-plane and  $\b_\alpha=-(0,0,1)$. Then there are $m,\theta\in\r$ such that
\begin{equation*}
\begin{split}
\vec{v}&=m(-\sin\theta,\cos\theta,1),\\
 \vec{b}&=k(\cos\theta,\sin\theta,0).
 \end{split}
 \end{equation*}
Both vectors in \eqref{c12} provide the parametrization of $\beta$  in \eqref{c9}. We now parametrize $\alpha$ by $\alpha(x)=(x,f(x),0)$ for some smooth function $f=f(x)$. Then 
$$\kappa_\alpha=\frac{f''}{(1+f'^2)^{3/2}},\quad \t_\alpha=\frac{(1,f',0)}{\sqrt{1+f'^2}}.$$ 
Then
$$\kappa_\alpha\b_\alpha-k\t_\alpha=\left(-\frac{k}{\sqrt{1+f'^2}},\frac{kf'}{\sqrt{1+f'^2}},-\frac{f''}{(1+f'^2)^{3/2}}\right).$$
The first equation of \eqref{c9} gives 
$$\frac{f''}{1+f'^2}=k\sin\theta+kf'\cos\theta.$$
A first integration of this equation yields \eqref{e12}.

\end{enumerate}

\end{enumerate}
\end{proof}

\begin{theorem}  \label{t32}
Let $M$ be a spacelike  translation surface \eqref{tt} such that  $\alpha''$  is timelike and $\beta''$ is a pseudo-null curve. If $M$ is a maximal surface, then $M$ is a plane or both curves are planar curves. In the latter case,   after a rigid motion the surface can be parametrized by
$$X(x,t)=(x,0,f(x))+\beta(t),$$
where we have two possibilities:
\begin{enumerate}
\item The function $f$ and the curve $\beta$ are given by  
\begin{equation}\label{e10}
\begin{split}
f'(x)&=-\tanh\left(m(x-f(x))+a\right),\\
\beta(t)&=m\frac{t^2}{2}(1,0,1)+t(b_1,1,b_1),
\end{split}
\end{equation}
where  $m,b_1,a\in\r$, $m\not=0$.  
\item The function $f$ and the curve $\beta$ are given by  
\begin{equation}\label{e13}
\begin{split}
f'(x)&=-\tanh\left(k(\cosh\theta f(x)-x\sinh\theta)+a\right), \\
\beta(t)&=e^{kt}m(\sinh\theta,1,\cosh\theta)-t(\cosh\theta,0,\sinh\theta),
\end{split}
\end{equation}
where  $a,k,\theta,m\in\r$, $k,m\not=0$.
\end{enumerate}
\end{theorem}

\begin{proof}
The argument is similar to the previous result and we only show the differences.
\begin{enumerate}
\item Case $k=0$. Then $\n_\beta=\vec{v}$ is constant and 
$$\beta(t)=\frac{t^2}{2}\vec{v}+\vec{b}t,$$
where $\langle \vec{v},\vec{b}\rangle=0$ and $\vec{b}$ is a unit spacelike vector. From \eqref{c8}, we have  $\langle\b_\alpha,\vec{v}\rangle=0$.  Differentiating with respect to $s$, we have $\tau_\alpha\langle\n_\alpha,\n\rangle=0$. If $\langle\n_\alpha,\vec{v}\rangle=0$, then $\vec{v}$ is parallel to $\t_\alpha$ which it is not possible. Thus $\tau_\alpha=0$ and, consequently, $\alpha$ is a planar curve. This proves Thm. \ref{t1} in this situation. Since the $\b_\alpha$ is spacelike,  after a rigid motion, we can assume that $\alpha$ is contained in the $xz$-plane and $\b_\alpha=(0,1,0)$. Let $\alpha(x)=(x,0,f(x))$, where $f\colon I\to\r$ is a smooth function. Hence we deduce 
\begin{equation*}
\begin{split}
\vec{v}&=m(1,0,1),\\
 \vec{b}&=(b_1,1,b_1),
 \end{split}
 \end{equation*}
where $m,b_1\in\r$, $m\not=0$. This gives the parametrization of $\beta$ in \eqref{e10}. 
The parametrization of the surface is
$$X(x,t)=(x,0,f(x))+m\frac{t^2}{2}(1,0,1)+t(b_1,1,b_1).$$
Notice that $\kappa_\alpha=\frac{f''}{(1-f'^2)^{3/2}}$ and the first equation of \eqref{c9} gives \eqref{e10}.  
\item Case $k\not=0$. From \eqref{pr11}, we have  
$$\beta(t)=e^{kt}\vec{v}-\frac{t}{k}\vec{b},$$
and \eqref{c9} is still valid. The case  $\langle\n_\alpha,\vec{b} \rangle=0$ is not possible and thus $\tau_\alpha=0$.   In particular, $\alpha$ is a planar curve. After a rigid motion, we write   $\alpha(x)=(x,0,f(x))$ where $f\colon I\to\r$ is a smooth function  and 
\begin{equation*}
\begin{split}
\vec{b}&=k(\cosh\theta,0,\sinh\theta),\\
\vec{v}&=m(\sinh\theta,1,\cosh\theta),
\end{split}
\end{equation*}
where $m,\theta\in\r$, $m\not=0$. The first equation of \eqref{c9} gives \eqref{e13} after a first integration.

\end{enumerate}

 \end{proof}

\subsection{Case both generating curves are pseudo-null curves}\label{s32}

Suppose that both generating curves are  pseudo-null curves. 

\begin{theorem}\label{t34}
 Let $M$ be a maximal translation surface parametrized by \eqref{tt} where  both two generating curves are pseudo-null curves. Then $M$ is a plane or $\alpha=\beta$ or, after a rigid motion, 
  \begin{equation*}
 \begin{split}
 \alpha(s)&=e^{ks}(0,1,1)-s(1,b,b),\\
 \beta(t)&=e^{kt}(w_1,w_2,w_3)-t( 1,b,b),
 \end{split}
 \end{equation*}
 where $w_2-w_3\not=0$ and 
 $$w_1= b(w_3-w_2) ,\quad b^2(w_2-w_3)+ w_2+w_3=0.$$
\end{theorem}

\begin{proof}
Equation \eqref{eq1} is 
$$(\t_\alpha,\t_\beta,\n_\alpha)+(\t_\alpha,\t_\beta,\n_\beta)=0.$$
Since $\t_\alpha\times\n_\alpha=-\n_\alpha$ and $\t_\beta\times\n_\beta=-\n_\beta$, this equation writes as
\begin{equation}\label{c3}
 \langle \n_\alpha,\t_\beta\rangle-\langle\t_\alpha,\n_\beta\rangle=0.
\end{equation}
Now the principal curvatures of $M$ are $\kappa_1=\langle \n_\alpha,\t_\beta\rangle$ and $\kappa_2=-\langle\t_\alpha,\n_\beta\rangle$. 
Notice that if $\alpha=\beta$, then the surface $X(s,t)=\alpha(s)+\alpha(t)$ is maximal for any pseudo-null curve $\alpha$.  From now, we assume $\alpha\not=\beta$.

Differentiating  \eqref{c3}with respect to $s$ and next with respect to $t$,  we obtain 
\begin{equation}\label{c6}
\begin{split}
0&=\kappa_\alpha\langle\n_\alpha,\t_\beta\rangle-\langle\n_\alpha,\n_\beta\rangle,\\
0&=\kappa_\alpha\langle\n_\alpha,\n_\beta\rangle-\kappa_\beta\langle\n_\alpha,\n_\beta\rangle.
\end{split}
\end{equation}
The second equation of \eqref{c6}  writes as
$$(\kappa_\alpha-\kappa_\beta)\langle\n_\alpha,\n_\beta\rangle=0.$$
\begin{enumerate}
\item Case $\kappa_\alpha-\kappa_\beta\not=0$ at some point $s=s_0$ and $t=t_0$. In an open set of $I\times J$ around $(s_0,t_0)$,  we have   $\langle\n_\alpha,\n_\beta\rangle=0$. The first equation of \eqref{c6} implies $\langle\n_\alpha,\t_\beta\rangle=0$ and consequently, from \eqref{c3}, we also have $\langle \t_\alpha,\n_\beta\rangle=0$. This proves that the principal curvatures of $M$ are identically $0$ and thus $M$ is contained in a plane. 
\item Case $\kappa_\alpha-\kappa_\beta=0$ identically. Since $\kappa_\alpha$ depends on $s$ and $\kappa_\beta$ on $t$, then $\kappa_\alpha=\kappa_\beta$ and $\kappa_\alpha$ and $\kappa_\beta$ are constant functions. Let   $\kappa_\alpha=\kappa_\beta=k$.
\begin{enumerate}
\item Case $k=0$. Then 
$$\alpha(s)=\frac{s^2}{2}\vec{v}+\vec{b}_1,\quad \beta(t)=\frac{t^2}{2}\vec{w}+t\vec{b}_2,$$
where $\vec{v},\vec{w}$ are lightlike vectors, $\vec{b}_1$ and $\vec{b}_2$ are unit spacelike vectors and orthogonal to $\vec{v}$ and $\vec{w}$, respectively. Then \eqref{c9} becomes
$$\langle\vec{v},\vec{w}\rangle s-\langle\vec{v},\vec{w}\rangle t+\langle \vec{b}_1,\vec{w}\rangle-\langle\vec{b}_2,\vec{v}\rangle=0.$$
This proves that $\langle\vec{v},\vec{w}\rangle=0$ and because  both vectors are lightlike, they must be proportional, $\vec{w}=\lambda\vec{v}$. Then it is immediate $\langle X(s,t),\vec{v}\rangle=0$ for all $s,t$. This proves that $M$ is contained in a plane.
\item Case $k\not=0$. Then
$$\alpha(s)=e^{ks}\vec{v}+s\vec{b_1},\quad \beta(t)=e^{kt}\vec{w}+t\vec{b}_2,$$
being $\vec{v}, \vec{w}$ and $\vec{b}_1$ and $\vec{b_2}$ as in the case $k=0$. Equation \eqref{c9} is now
$$e^{ks}\langle\vec{b}_2,\vec{v}\rangle-e^{kt}\langle \vec{b}_1,\vec{w}\rangle=0.$$
This gives $\langle\vec{b}_2,\vec{v}\rangle=\langle \vec{b}_1,\vec{w}\rangle=0$. Moreover,   the condition 
  $\t\times\n =-\n $  implies $\vec{b}_1\times\vec{v}=-\vec{v}$ and $\vec{b}_2\times\vec{w}=-\vec{w}$.

After a rigid motion, we   fix the curve $\alpha$. Taking into in account that $\t_\alpha\times\n_\alpha=-\n_\alpha$, we assume without loss of generality that
$$\vec{v}=(0,1,1),\quad \vec{b}_1=(1,b,b).$$
We have two possibilities:
\begin{enumerate}
\item Case $b=0$. Then $\vec{b}_2=(-1,0,0)$ and $\vec{w}=(0,w_2,-w_2)$.
\item Case $b\not=0$. Then  $w_2-w_3\not=0$, $\vec{b}_2=(-1,-b,-b)$ and $\vec{w}=(b(w_3-w_2),w_2,w_3)$, where  $b^2(w_3-w_2)-w_2-w_3=0$.
\end{enumerate}
\end{enumerate}
\end{enumerate}
\end{proof}

\begin{example} The following two surfaces are obtained from Thm. \ref{t34}. 
\begin{enumerate}
\item Let $b=0$ and $w_2=-w_3=1$. Then 
$$X(s,t)=(-s+t,e^{ks}+e^{kt},e^{ks}-e^{kt}),\quad s,t\in\r.$$
\item Let $b=1$ and $w_2=0$ and $w_3=1$. Then 
$$X(s,t)=( e^{kt}-s-t, e^{ks}-s-t, e^{ks}+e^{kt}-s-t),\quad s,t\in\r.$$
\end{enumerate}

\end{example}

 \section{The generating curves are of Frenet type: general case}\label{s4}
 
 In this section we consider the case the both generating curves are of Frenet type.    The equation $H=0$ given \eqref{eq1} is now
 \begin{equation}\label{eq40}
 \kappa_\alpha\langle\b_\alpha,\t_\beta\rangle=\kappa_\beta\langle\t_\alpha,\b_\beta\rangle.
 \end{equation}
 Differentiating with respect to $s$ and using \eqref{eq40}, we have
\begin{equation}\label{eq4}
\langle \tau_\alpha\n_\alpha +\frac{\kappa_\alpha'}{\kappa_\alpha}  \b_\alpha,\t_\beta\rangle=\kappa_\beta\langle\n_\alpha,\b_\beta\rangle.
\end{equation}
 
From now on, we may suppose that the generating curves $\alpha$ and $\beta$ are non planar. We need to introduce the following notation. For a non-planar curve parameterized by arc length  with curvature $\kappa$ and torsion $\tau$, let
\begin{equation}\label{eqR}
R=\frac{\kappa'}{\kappa}+\frac{\tau'}{\tau},\quad \Sigma=\left(\frac{\kappa'}{\kappa}\right)'+\kappa^2+\tau^2.
\end{equation}
The subscript $\alpha$ or $\beta$ in $R$ and $\Sigma$ will indicate  the  curve $\alpha$ or $\beta$. 

Differentiating   (\ref{eq4}) with respect to   $s$, taking into account  (\ref{eq40}), and using the notation \eqref{eqR}, we obtain
$$\langle-\epsilon_\alpha\kappa_\alpha\t_\alpha+R_\alpha\n_\alpha+\frac{\Sigma_\alpha}{\tau_\alpha} \b_\alpha,\t_\beta\rangle=\kappa_\beta\langle\b_\alpha,\b_\beta\rangle.$$
If $U(s)=-\epsilon_\alpha\kappa_\alpha\t_\alpha+R_\alpha\n_\alpha+\frac{\Sigma_\alpha}{\tau_\alpha} \b_\alpha$,
the above equation is 
$$\langle U,\t_\beta\rangle=\kappa_\beta\langle\b_\alpha,\b_\beta\rangle.$$
Differentiating again with respect to $s$ and using   \eqref{eq40} and \eqref{eq4}, we get
$$\langle U'-\tau_\alpha\frac{\kappa_\alpha'}{\kappa_\alpha}-\tau_\alpha^2\n_\alpha,\t_\beta\rangle=0.$$
Let 
$$u=u(s)=U'-\tau_\alpha\frac{\kappa_\alpha'}{\kappa_\alpha}-\tau_\alpha^2\n_\alpha.$$
Differentiating twice with respect to $t$, we conclude $\langle u,\n_\beta\rangle=0$ and $\langle u,\b_\beta\rangle=0$. Thus $u=0$. Computing the coordinates of $u$ with respect to the Frenet frame of $\alpha$, we have
\begin{equation}\label{ta}
\begin{split}
 \kappa_\alpha'+\kappa_\alpha R_\alpha&=0,\\
 R_\alpha'+\Sigma_\alpha-\epsilon_\alpha\kappa_\alpha^2-\tau_\alpha^2&=0,\\
 \left(\frac{\Sigma_\alpha}{\tau_\alpha}\right)'+R_\alpha\tau_\alpha-\tau_\alpha\frac{\kappa_\alpha'}{\kappa_\alpha}&=0.
 \end{split}
 \end{equation}
 Similar equations are obtained for the curve $\beta$.
 
 \begin{theorem}\label{t41}
 Let $M$ be a maximal  translation surface where the generating curves are of Frenet type. If both curves are non-planar, then 
 \begin{equation}\label{eq88}
\kappa_\alpha^2\tau_\alpha=c_1\not=0,\ \kappa_\beta^2\tau_\beta=\bar{c}_1\not=0,\ \frac{\Sigma_\alpha}{\tau_\alpha}+\tau_\alpha=c_2,\ \frac{\Sigma_\beta}{\tau_\beta}+\tau_\beta=\bar{c}_2,
\end{equation}
 where $c_1$, $c_2$, $\bar{c}_1$ and $\bar{c}_2$ are constants.
 \end{theorem}

 \begin{proof}  The next argument is valid for both curves. From the first equation of \eqref{ta} and the definition of $R$ we have 
\begin{equation}\label{eq18}
R=-\frac{\kappa'}{\kappa}.
\end{equation}
 By the definition of $R$, we obtain  $2\kappa'\tau+\kappa\tau=0$ which means that the derivative of $\kappa^2\tau$ is $0$, hence $\kappa^2\tau$ is constant. From the third equation of \eqref{ta} together \eqref{eq18}, we have 
 $$\left(\frac{\Sigma}{\tau}+\tau\right)'=-\tau R+\tau\frac{\kappa'}{\kappa}+\tau'=2\tau\frac{\kappa'}{\kappa}+\tau'=0.$$
 This proves that $\Sigma/\tau+\tau$ is a constant function.
  
\end{proof}

A curve with constant curvature and constant torsion satisfies the first condition in \eqref{eq88} and it is a candidate to be a generating curve of a maximal translation surface. For this reason, we focus on circular helices.   By a circular helix in $\rr^3$ we mean a Frenet curve with constant curvature and constant torsion. These curves are invariant by a uniparametric group of helicoidal motions whose axis is timelike or spacelike. In $\rr^3$ there are different types of circular helices, depending on the group of rotations \cite{lo2}. When these helices are spacelike curves, we have:  
\begin{enumerate}
\item Helices of type I. If $r^2>h^2>0$, then
\begin{equation}\label{h1}
\alpha(s)=(r\cos \varphi(s),r\sin \varphi(s),h \varphi(s)),\quad \varphi(s)=\frac{s}{\sqrt{r^2-h^2}}.
\end{equation}
\item Helices of type II. For any $r,h\not=0$, we have 
\begin{equation}\label{h2}
\alpha(s)=(h\varphi(s),r\sinh \varphi(s),r\cosh \varphi(s)),\quad \varphi(s)=\frac{s}{\sqrt{h^2+r^2}}.
\end{equation}
\item Helices of type III. If $h^2>r^2>0$, then
\begin{equation}\label{h3}
\alpha(s)=(h\varphi(s),r\cosh \varphi(s),r\sinh \varphi(s)),\quad \varphi(s)=\frac{s}{\sqrt{h^2-r^2}}.
\end{equation}
\end{enumerate}

All these curves are parametrized by arc-length. After minor modifications on the inequalities between $r$ and $h$, we obtain the corresponding timelike circular helices.  For example, for the helix  of type I, if we replace $\varphi$ by $\varphi(s)=\frac{s}{\sqrt{h^2-r^2}}$ we obtain a timelike curve with constant curvature and constant torsion.

For a circular helix, the curvature is constant. Thus the following result is immediate from \eqref{eq40}.  

\begin{proposition} \label{pr41}
If  $\alpha$ is a circular spacelike helix, then the  surface  
$\Psi(s,t)=\alpha(s)+ \alpha(t)$ has zero mean curvature everywhere.
\end{proposition}

As we have pointed previously, this surface can be have regions that are spacelike and other that are timelike, hence we have stated the proposition by saying that the surface has zero mean curvature instead to say that it is maximal. We give an example showing this behaviour on the causal character of the points of the surface.

\begin{example} Consider a helix of type II. Let $r=1$ and let   $\alpha(s)=(hs,\cosh s,\sinh s)$, $h>0$. Then the surface $\Psi(s,t)$ is degenerated only when $s=t$. If $h=1$, the surface is timelike everywhere.  However, for $h=2$, the surface has regions that are spacelike and timelike.
\end{example}

In the following result, we prove that if a generating curve of a  maximal translation surface is a circular helix, then the other one is congruent with the given helix.

\begin{theorem}\label{t42} Let $M$ be a maximal translation surface. If one of the generating curves is a circular helix, then the other curve is a translation of the same helix. \end{theorem}

\begin{proof} Assume that $M$ is parametrized by \eqref{tt} and that   $\alpha$   a circular helix. We prove that $\beta$ is congruent to $\alpha$.  Although there are three types of circular helices,  the arguments are similar in all cases. We do the proof when   $\alpha$ is   of type I,   \eqref{h1}.   The computation of \eqref{eq40} gives a polynomial of type $A_0(t)+A_1(t)\sin s+A_2(t)\cos s=0$. The coefficients $A_i$ are
\begin{equation*}
\begin{split}
A_2&=h\beta_2'-(h^2-r^2)(\beta'\times\beta'')_2,\\
A_1&=h\beta_1'-(h^2-r^2)(\beta'\times\beta'')_1,\\
A_0&=-r^2\beta_3'+h(h^2-r^2)(\beta'\times\beta'')_3,
\end{split}
\end{equation*}
where the subindex in $\beta$ and $\beta'\times\beta''$ indicates the corresponding coordinates in  $\rr^3$. Combining the three equations, we have
\begin{equation*}
\begin{split}
h\beta_1'^2+h\beta_2'^2&=(h^2-r^2)(\beta_2'(\beta'\times\beta'')_2+\beta_1'(\beta'\times\beta'')_1)\\
&=(h^2-r^2)\beta_3'(\beta'\times\beta'')_3= \frac{r^2\beta_3'^3}{h}
\end{split}
\end{equation*}
Thus
$$\frac{r^2\beta_3'^3}{h}=h\beta_1'^2+h\beta_2'^2=h+h\beta_3'^2=h(1+\beta_3'^2).$$
 Then $\beta_3'^2=\frac{h^2}{r^2-h^2}$, and after a vertical translation of $\beta$, we can assume 
 $$\beta_3(t)=\frac{h}{\sqrt{r^2-h^2}}t.$$
 Hence
 $$\beta_1'^2+\beta_2'^2=1-\beta_3'^2= \frac{r^2}{r^2-h^2}.$$
 Then there is $\theta=\theta(t)$ such that 
 $$\beta_1'(t)=-\frac{r}{\sqrt{r^2-h^2}}\sin\theta(t),\quad \beta_2'(t)=\frac{r}{\sqrt{r^2-h^2}}\cos\theta(t).$$
 Equation $A_0=0$ gives 
 $$\theta'(t)=\frac{1}{\sqrt{r^2-h^2}}.$$
 After a translation on the parameter $s$, we have $\theta(t)=\varphi(t)$. Then we deduce that $\beta(t)=\alpha(t)$ up to a vertical translation.
 
\end{proof}

 \section{The generating curves are of Frenet type: a curve is planar}\label{s5}
 
 In this section, we study the case that the two generating curves are   Frenet curves and one of them is contained in a plane of $\rr^3$.  We discard the trivial case that the surface is a plane. Without loss of generality, we assume that $\alpha$ is planar. Then  $\tau_\alpha=0$.  We repeat the computations of the beginning of Sect. \ref{s4} obtaining
\begin{equation}\label{eq5}
\left(\left(\frac{\kappa_\alpha'}{\kappa_\alpha}\right)'+\epsilon_\alpha\kappa_\alpha^2\right) \langle  \b_\alpha,\t_\beta\rangle=0.
\end{equation}
If $\langle\b_\alpha,\t_\beta\rangle=0$ identically, then the principal curvature $\kappa_1$ of $M$ is $0$ identically and this proves that   $M$ is a plane which was initially discarded. Thus 
\begin{equation}\label{kk}
\left(\frac{\kappa_\alpha'}{\kappa_\alpha}\right)'+\epsilon_\alpha\kappa_\alpha^2=0.
\end{equation}
The solution of this equation depends if $\epsilon_\alpha$ is $-1$ or $1$. We give the  solutions of \eqref{kk} in the following proposition. We drop the dependence on $\alpha$.

\begin{proposition} \label{pr3}
The solutions of Eq. \eqref{kk} are:
\begin{enumerate} 
\item Case $\epsilon=1$. Then the solution  is $\kappa(s)=\frac{c}{\cosh(cs)}$ and the curve is given in \eqref{cu1}. 
\item Case $\epsilon=-1$. We have three types of solutions.
\begin{enumerate}
\item Case $\kappa(s)=\frac{c}{\cos(cs)}$  and the curve is given in \eqref{cu2}.
\item Case $\kappa(s)=\frac{c}{\sinh(cs)}$ and the curve is given in \eqref{cu3}. 
\item Case $\kappa(s)=\frac{1}{s}$. Then $\alpha$ is parametrized by  
$$\alpha(s)=\frac{1}{2}\left(\frac{s^2}{2}+\log(s),0, \frac{s^2}{2}-\log(s)\right).$$
\end{enumerate}
\end{enumerate}
\end{proposition}
\begin{proof}
The case $\epsilon=1$  is known, see \cite{hl}.

Let now $\epsilon=-1$.    Define the function $z=z(\kappa)=\sqrt{\kappa'}$.   Then \eqref{kk} is $2\kappa z^3z'=z^4+\kappa^4$. Let $v=z/\kappa$, obtaining 
$2\kappa v^3v'=1-v^4$. A particular solution is $v=\pm 1$. This gives $\kappa(s)=1/s$. If $v^2\not=1$, then we have two solutions, $\kappa(s)= \frac{c}{\cos(cs)}$ and $\kappa(s)= \frac{c}{\sinh(cs)}$. This gives the three solutions given in  (2). 

Since the curve is contained in a timelike plane, we can assume that this plane is the $xz$-plane. Suppose that we parametrize the curve  as 
$\alpha(s)=(x(s),0,z(s))$. Then there is a smooth function $\phi(s)$ such that  $x'(s)=\cosh\phi(s)$, $z'(s)=\sinh\phi(s)$ and $\phi'(s)=\kappa(s)$. The curves that are solutions appear   in the statement of the proposition. 
\end{proof}

We now prove the first statement of Thm. \ref{t1} when both curves are of Frenet type, that is, if a curve of a maximal translation surface is planar, then the other is also planar. 

\begin{theorem}\label{t43} 
Let $M$ be a maximal translation surface whose generating curves are of Frenet type.  If a curve is planar, then the other one is also planar.
\end{theorem}

\begin{proof} We will use that the property that if a curve $\beta=\beta(t)$ satisfies $(\beta'(t), \beta''(t),\beta'''(t))=0$ for all $t$, then $\beta$ is contained in a plane.  We have different cases.
\begin{enumerate}
\item Case $\epsilon_\alpha=1$.  After a rigid motion, we can assume that $\alpha$ is contained in the $xy$-plane and  $\alpha(x)=(x,\frac{1}{c}\log(\cos (c x)),0)$.

 Let   $\beta(t)=(\beta_1(t),\beta_2(t),\beta_3(t))$ denote the  other generating curve   parameterized by arc length $t$. Then  the minimality condition (\ref{eq1}) gives
$$(\beta_1'\beta_3''-\beta_1''\beta_3')\sin(cx)+(c\beta_3'+\beta_2'\beta_3''-\beta_2''\beta_3')\cos(cx)=0.$$
Since the functions $\sin(cx)$ and $\cos(cx)$ are linearly independent, we deduce
\begin{equation}\label{eqbb}
\begin{split}
0&=\beta_1'\beta_3''-\beta_1''\beta_3',\\
0&= c\beta_3'+\beta_2'\beta_3''-\beta_2''\beta_3'.
\end{split}
\end{equation}
Combining both equations, we have $\beta_3'(-\beta_1''\beta_2'+c\beta_1'+\beta_2'\beta_1'')=0$. If $\beta_3$ is a constant function, then the surface is contained in the plane $z=\beta_3$.  On the contrary, $\beta_1''\beta_2'=c\beta_1'+\beta_2'\beta_1''$. From the first equation in (\ref{eqbb}), we obtain $\beta_1'\beta_3'''=\beta_1'''\beta_3'$. With these relations,    the determinant  $(\beta',\beta'',\beta''')$ is $0$. Therefore, $\beta$ is a planar curve and  $\tau_\beta=0$.

\item Case $\epsilon_\alpha=-1$.  We need to distinguish cases. The cases $\alpha(x)=(x,0,\frac{1}{c}\log(\cosh(cx)))$ and  $\alpha(x)=(0,\frac{1}{c}\log(\sinh(cz)),z)$ are similar to the case $\epsilon_\alpha=1$ and we omit the proof. 

We   prove the third case. So, suppose $\alpha(s)=\frac{1}{2}\left(\frac{s^2}{2}+\log(s),0, \frac{s^2}{2}-\log(s)\right)$.    Let $\beta(t)=(\beta_1(t),\beta_2(t),\beta_3(t))$. Then $H=0$ is a polynomial on $s$, $A_0(t)+A_2(t)s^2=0$. Equations $A_0=A_2=0$ are
\begin{equation*}
\begin{split}
0&=  2\beta_2'+\beta_2'\beta_1''-\beta_1'\beta_2''+\beta_2'\beta_3''-\beta_3'\beta_2''\\
0&=\beta_1'\beta_2''-\beta_2'\beta_1''+\beta_2'\beta_3''-\beta_3'\beta_2''.
\end{split}
\end{equation*}
Differentiating with respect to $t$, we get 
$$-\beta_2''=(\beta'\times\beta''')_1=(\beta'\times\beta''')_3.$$
Looking both equations on $\beta_2'$ and $\beta_2''$, there must no trivial solutions because, otherwise, $\beta_2$ is constant and the surface would be a plane of equation $y=\beta_2$. Thus, the determinant of coefficients of these equations must vanish identically, obtaining 
$$\beta_1'-\beta_3'+\beta_1'\beta_3''-\beta_3'\beta_1''=0.$$
This gives $\beta_1''-\beta_3''=(\beta'\times\beta''')_2$. Then
$$\langle \beta'',\beta'\times\beta'''\rangle=(\beta_1''-\beta_3'')(\beta'\times\beta''')_1+\beta_2'' (\beta'\times\beta''')_2 =0.$$
This proves that the curve is planar. 
 \end{enumerate}
\end{proof}

From Thm. \ref{t43}, we know that both generating curves are planar. The remaining of this section is devoted to calculate the parametrizations of these surfaces.  Since we know that $\beta$ is planar, a similar computation as it was done for $\alpha$ implies that the curvature $\kappa_\beta$ satisfies
$$\left(\frac{\kappa_\beta'}{\kappa_\beta}\right)'+\epsilon_\beta\kappa_\beta^2=0.$$
  We distinguish all cases $\epsilon_\alpha,\epsilon_\beta\in\{-1,1\}$. So, there are $10$ cases to discuss. As we will see, in many of them, there are not translation maximal surfaces. See Table \ref{table} as a summary of all cases.
  
  The strategy in the proofs  repeats in each case and it is bit tedious. We parametrize a maximal surface $M$ by \eqref{tt} and the curve $\alpha$ will be fixed and parametrized according to Prop. \ref{pr3}. The curve $\beta$ is also planar and thus, up to a rigid motion $A$, $\beta$ is also parametrized by the expression of  Prop. \ref{pr3}. We will write $\beta(t)=A\cdot\sigma(t)$, where $\sigma$ is a curve parametrized by Prop. \ref{pr3}. Let   $A=(a_{ij})$. We will use frequently that $A$ is orthogonal with respect to this Minkowski metric. For example,  the rows as well as the columns form a orthonormal basis of $\rr^3$. Also, if we denote the rows of  $A$ by   $A=[A_1,A_2,A_3]$, then $A_1\times A_2=-A_3$ and $A_2\times A_3=A_1$. 

We will obtain the parametrizations that appear in Thm. \ref{t2}.  Recall that a summary appears in Table  \ref{table}.

\subsection{Case $\epsilon_\alpha=1$.}
The parametrization of $\alpha$ is $\alpha(s)=(s,\frac{1}{c}\log(\cos(cs)),0)$.
\begin{enumerate}
\item Case $\epsilon_\beta=1$. Then  $\beta(t)=A\cdot\sigma(t)$ where $\sigma(t)=(t,\frac{1}{d}\log\cos(dt),0)$. The minimality condition \eqref{eq1} is an equation of type $B_1(t)\sin(cs)+B_2(t)\cos(cs)=0$. Thus $B_1=B_2=0$. From $B_1=0$, we obtain $a_{12} a_{31}-a_{11} a_{32}=0$, hence $a_{23}=0$.
The coefficient $B_2$, once multiplied by $\cos(dt)$ is a polynomial $C_0+C_1\sin(dt)+C_2\cos(dt)=0$. Using that $A$ is an isometry, the equation $C_1=0$ becomes $a_{32}=0$.  Thus $a_{12}a_{31}=0$. 
\begin{enumerate}
\item Case $a_{12}=0$. Then $a_{22}=1$. In particular, $a_{21}=0$. Then $C_2=0$ reduces into $a_{31}(c+d)=0$. If $a_{31}=0$, then $A$ is the diagonal matrix $[1,1,-1]$, and $\beta=\alpha$. In particular, the third coordinate is $0$, proving that the surface is contained in the plane $z=0$. Otherwise, $d=-c$ and we have    $$A=\begin{pmatrix}\cosh\theta&0&\sinh\theta\\ 0&1&0\\ \sinh\theta&0&\cosh\theta\end{pmatrix}.$$
This gives the parametrization \eqref{para1}.
\item Case $a_{13}=0$. Then $a_{33}=1$ and  $a_{13}=0$. 
This implies that the surface is included in the plane $z=0$, so it is a plane. 
\end{enumerate}

\item Case $\epsilon_\beta=-1$.
  We know that $\beta(t)=A\cdot \sigma(t)$, where  $\sigma$ is parametrized by three types.

\begin{enumerate}
 
 \item Case   $\sigma(t)=(t,0,\frac{1}{d}\log(\cosh (dt))$.  Then \eqref{eq1} is    $B_1(t)\sin(cs)+B_2(t)\cos(cs)=0$. Equation $B_1=0$ gives $a_{11}a_{33}-a_{13}a_{31}=0$ and thus, $a_{22}=0$. Imposing $B_2(t)=0$, we have a polynomial equation of type $C_0+C_1 \sinh(dt)+C_2\cosh(dt)=0$. From  $C_1=0$, and using that $A$ is an isometry, we obtain   $a_{33}=0$, which it is not possible. This  case cannot occur.
 \item Case   $\sigma(t)=(0, \frac{1}{d}\log\sinh(dt),t)$ and $d>0$ is a positive constant.  Then Eq. \eqref{eq1} is  $B_1(t)\sin(cs)+B_2(t)\cos(cs)=0$ and $B_1=0$ is $a_{13} a_{32}-a_{12} a_{33}=0$. Thus $a_{21}=0$. Now $B_2=0$ is
 $$ ca_{32}\coth(dt)+a_{33}(c+da_{22})-d a_{23}a_{32}=0.$$
 This implies $a_{32}=0$  and $c+da_{22}=0$. Thus $a_{12}=0$. This gives $d=-c$ and 
 $$A=\left(
\begin{array}{ccc}
 \cosh \theta & 0 & \sinh \theta  \\
 0 & 1 & 0 \\
 \sinh \theta  & 0 & \cosh \theta  \\
\end{array}
\right).$$
  This gives the parametrization \eqref{para2}.

\item Case  $\sigma(t)=\frac{1}{2}\left(\frac{t^2}{2}+\log(t),0, \frac{t^2}{2}-\log(t)\right)$. Then \eqref{eq1} is   a $B_1(t)\sin(cs)+B_2(t)\cos(cs)=0$. Equation $B_1=0$ is
$$a_{13} a_{32}-a_{12} a_{33}=0.$$
Again, we have   $a_{21}=0$. Simplifying $B_2=0$, we get 
$$a_{11}(8+4c(a_{23}-a_{22}))t^2+4a_{11}a_{22}+a_{23})t^4=0.$$
Since the second row $A_2=(0,a_{22},a_{23})$ is a unit spacelike vector, then $a_{22}+a_{23}\not=0$. Thus $a_{11}=0$. This implies $a_{33}=0$, which it is not possible. This case cannot occur. 
\end{enumerate}

\end{enumerate}

\subsection{Case $\epsilon_\alpha=-1$ and $\alpha(s)=(s,0,\frac{1}{c}\log\cosh(c s))$}
\begin{enumerate}
\item   Case  $\beta(t)=A\cdot\sigma(t)$, where  $\sigma(t)=(t,0, \frac{1}{d}\log\cosh(dt))$ and $d>0$.  Then Eq. \eqref{eq1} is  $B_1(t)\sinh(cs)+B_2(t)\cosh(cs)=0$. Now $B_1=0$ is $a_{13} a_{21}-a_{11} a_{23}=0$. 
This implies $a_{32}=0$. If $a_{11}=0$, and because $a_{21}$ cannot be $0$, then $a_{13}=0$. This proves that 
$$A=\begin{pmatrix}0&1&0\\ \cosh\theta&0&\sinh\theta\\ \sinh\theta&0&\cosh\theta\end{pmatrix}.$$
Then $B_2=0$ is $(c\cosh\theta+d)\cosh(dt)+c\sinh\theta\sinh(dt)=0$. This gives $d=-c$ and $\theta=0$. Then the surface is the plane of equation $z=0$.

Otherwise, by using that $A_2\times A_3=A_1$, we have $a_{11}=a_{22}a_{23}$, $a_{13}=a_{22}a_{31}$, $a_{21}=-a_{12}a_{33}$ and $a_{23}=-a_{12}a_{31}$. Then $B_2=0$ becomes $a_{12}\left((ca_{33}+ d)\cosh(dt)+a_{31}c\sinh(dt)\right)=0$. If $a_{12}=0$, then the surface is the plane of equation $y=0$, which it is not spacelike. Thus $a_{31}=0$, $a_{33}=1$ and $d=-c$. Then 
$$A=\begin{pmatrix}\cos\theta&\sin\theta&0\\ -\sin\theta&\cos\theta&0\\ 0&0&1\end{pmatrix}.$$
This gives the parametrization \eqref{para3}.
  
\item Case   $\sigma(t)=(0, \frac{1}{d}\log\sinh(dt),t)$ and $d>0$.  Then Eq. \eqref{eq1} is  $B_1(t)\sinh(cs)+B_2(t)\cosh(cs)=0$. Now $B_1=0$ is $a_{13} a_{22}-a_{12} a_{23}=0$. 
This implies $a_{31}=0$. The equation $B_2=0$ writes as 
$$a_{22}c\coth(dt)+a_{22}a_{33}+a_{23}(c-a_{32}d)=0.$$
This implies $a_{22}=0$. Since $a_{12}a_{23}=0$ and the second column of $A$ is $(a_{12},0,a_{32})$ with $a_{32}\not=1$, then $a_{23}=0$. Then the surface is included in the plane $y=0$, which it is not possible. This case cannot occur. 

  \item Case  $\sigma(t)=\frac{1}{2}\left(\frac{t^2}{2}+\log(t),0, \frac{t^2}{2}-\log(t)\right)$. Then \eqref{eq1} is    $B_1(t)\sinh(cs)+B_2(t)\cosh(cs)=0$. Equation $B_1=0$ is $a_{13} a_{21}-a_{11} a_{23}=0$. 
This implies $a_{32}=0$, $a_{12}=a_{22}a_{33}$ and $a_{13}=a_{22}a_{31}$. If $a_{12}=0$, then the surface is included in the plane $y=0$, which it is not possible.  Since the first row $A_1$ of $A$ is a unit spacelike vector then $a_{12}^2+a_{22}^2=1$.  Now we conclude $a_{22}=0$. Thus 
$$A=\begin{pmatrix}0&1&0\\ \cosh\theta&0&\sinh\theta\\ \sinh\theta&0&\cosh\theta\end{pmatrix}.$$
Then $B_2=0$ writes as $2+c e^{-\theta}+ce^{\theta}t^2=0$. This gives $c=0$, a contradiction. Thus this case is not possible.

\end{enumerate}

\subsection{Case $\epsilon_\alpha=-1$ and $\alpha(s)=(0,\frac{1}{c}\log\sinh(c s),s)$}
\begin{enumerate}

\item Case   $\sigma(t)=(0, \frac{1}{d}\log\sinh(dt),t)$ and $d>0$.  Equation \eqref{eq1} writes as $B_0(t)+B_1(t)\sinh(cs)+B_2(t)\cosh(cs)=0$. Equation $B_1=0$ is $a_{12} a_{23}-a_{13} a_{22}=0$. Thus $a_{31}=0$ and $a_{21}$, hence $a_{12}=a_{13}=0$. Then  the surface is included in the plane $x=0$, which it is not possible because this plane s timelike.

  \item Case  $\sigma(t)=\frac{1}{2}\left(\frac{t^2}{2}+\log(t),0, \frac{t^2}{2}-\log(t)\right)$.  Equation \eqref{eq1} is $B_0(t)+B_1(t)\sinh(cs)+B_2(t)\cosh(cs)=0$. Now $B_1=0$ gives $a_{11}a_{23}-a_{13}a_{21}=0$. Then $a_{32}=0$. If $a_{22}=0$, then the surface is the plane $x=0$, which it is not possible. The coefficient $B_2$ is a polynomial on $t$. Letting $B_2=0$, we obtain
   $(2+c(a_{33}-a_{31}))+c(a_{31}+a_{33})t^2=0$. This gives $a_{31}=-a_{33}$ and the file $A_3$ of $A$ would be lightlike, a contradiction. This case cannot occur.

\end{enumerate}

\subsection{Case $\epsilon_\alpha=-1$ and $\alpha(s)=\frac12(\frac{s^2}{2}+\log(s),0,\frac{s^2}{2}-\log(s))$}

Assume  $\beta(t)=A\cdot\alpha(t)$.  Equation \eqref{eq1} is a polynomial on $s$ of type $B_0(t)+B_2(t)s^2=0$. The equation $B_2=0$ writes is
$$a_{11}a_{23}-a_{13}a_{21}+a_{21}a_{33}-a_{23}a_{31}=0.$$
This implies $(A_3)_2-(A_1)_2=0$, that is, $a_{32}=a_{22}$, hence $a_{12}=1$. If $a_{22}=0$, then 
\begin{equation}\label{a3}
A=\begin{pmatrix}0&1&0\\ \cosh\theta&0&\sinh\theta\\ \sinh\theta&0&\cosh\theta\end{pmatrix}.
\end{equation}
Now $B_2=4t^2$, obtaining a contradiction. 

Suppose $a_{22}\not=0$. Then $a_{12}=1$. Then 
$$A=\left(
\begin{array}{ccc}
 a_{11} & 1 &a_{11} \\
 a_{21}& a_{22} & a_{23}\\
a_{31}&a_{22} &a_{33}\\
\end{array}
\right).$$
But this matrix cannot be orthogonal.

 \begin{table}[hbtp]
 \begin{tabular}{|c|c|c|c|c|}
\hline
\diagbox{$\alpha$}{$\beta$}  & $\epsilon=1$  & $\epsilon=-1$ (a) & $\epsilon=-1$ (b) & $\epsilon=-1$ (c) \\
\hline
$\epsilon=1$ &  Eq. \eqref{para1} & X & Eq. \eqref{para2}& X\\
\hline
$\epsilon=-1$ (a) &   & Eq. \eqref{para3} & X & X\\
\hline
$\epsilon=-1$ (b) &   &  &X & X \\
\hline
$\epsilon=-1$ (c)&   & &  & X\\
\hline
\end{tabular}
\caption{Summary of the family of Scherk surfaces of Thm. \ref{t1} with the corresponding parametrizations.   The symbol X means that the case cannot happen.}\label{table}
\end{table}

\section{Acknowledgement} The   author   has been partially supported by MINECO/ MICINN/FEDER grant no. PID2023-150727NB-I00, and by the ``Mar\'{\i}a de Maeztu'' Excellence Unit IMAG, reference CEX2020-001105- M, funded by MCINN/AEI/ 10.13039/501100011033/ CEX2020-001105-M.



\begin{thebibliography}{00}

 \bibitem{ay} M. E. Aydin, Classifications of translation surfaces in isotropic geometry with constant curvature. Ukranian Math. J. 72 (2020), 329--347.
 
 \bibitem{da} L. C. B. Da Silva,   Moving frames and the characterization of curves that lie on a surface. J. Geom. 108 (2017),   1091--1113.
 
\bibitem{di} F. Dillen, I. Van de Woestyne, L. Verstraelen and J. T. Walrave, The surface of Scherck in $E^3$: a special case in the class of minimal surfaces defined as the sum of two curves. Bull. Inst. Math. Acad. Sin., 26 (1998), 257--267.


\bibitem{hl} T. Hasanis, R. L\'opez, Classification and construction of minimal translation surfaces in Euclidean space. Results Math. 75 (2020), no. 1, Paper No. 2, 22 pp.

\bibitem{hl2} T. Hasanis, R. L\'opez, Translation surfaces in Euclidean space with constant Gaussian curvature.  Comm. Anal. Geom. 29 (2021),  1415-- 1447.



\bibitem{HLin} H. Liu, Translation surfaces with constant mean curvature in 3-dimensional spaces. J. Geom. 64 (1999), 141--149.

\bibitem{Lopez} R. L\'opez, Minimal translation surfaces in hyperbolic space. Beitr. Algebra Geom, 52 (2011), 105--112.

\bibitem{lo2} R. L\'opez,  Differential Geometry of curves and surfaces in Lorentz-Minkowski space. Int. Electron. J. Geom.   7 (2014), 44--107.

\bibitem{lo22} R. L\'opez,   Surfaces in Lorentz-Minkowski space with mean curvature and Gauss curvature both constant. Differential Geometry in Lorentz-Minkowski space, 71--85, Ed. Univ. Granada, Granada, 2017. 

\bibitem{lomu} R. L\'opez, M. I.  Munteanu, Surfaces with constant mean curvature in Sol geometry. Differential Geom. Appl. 29 (2011), suppl. 1, S238--S245.

\bibitem{lp}  R. L\'opez and O. Perdomo,   Minimal translation surfaces in Euclidean space. J. Geom. Anal. 27 (2017), 2926--2937.

\bibitem{mi} Z. Milin Sipus,   Translation surfaces of constant curvatures in a simply isotropic space. Period. Math. Hung. 68 (2014), 160--175.

\bibitem{mr} S. Montiel, A. Ros, Curves and Surfaces. Graduate Studies in Mathematics
Volume 69,  American Mathematical Society, 2009.

 

\bibitem{ni} J. C. C. Nitsche, Lectures on Minimal Surfaces. Cambridge Univ. Press. Cambridge, 1989.

\bibitem{sc} H. F. Scherk, Bemerkungen \"{u}ber die kleinste Fl\"{a}che innerhalb gegebener Grenzen. J. Reine Angew. Math. 13 (1835), 185--208.

\bibitem{wo} I. Van de Woestyne, Minimal surfaces of the 3-dimensional Minkowski space. M. Boyom, J.-M. Morvan and L. Verstraelen, editors, Geometry and Topology of Submanifolds II. World Scientific Publishing, Singapore, 1990, 244--369

\end{thebibliography}
\end{document}